\documentclass{article}
\usepackage[utf8]{inputenc}
\usepackage[]{amsfonts}
\usepackage{amssymb,latexsym,amsmath,epsfig,mathrsfs}
\usepackage{amsthm}
\usepackage[all]{xy}
\usepackage{graphicx}
\usepackage{float}
\usepackage{textcomp}
\usepackage{bm}
\usepackage{mathtools}
\usepackage{tikz}
\usetikzlibrary{matrix,arrows,decorations.pathmorphing}
\usepackage{tikz-cd}
\usepackage{listings}
\usepackage{fullpage}
\usepackage{microtype}
\usepackage{colonequals}
\usepackage{hyperref}
\usepackage[marginpar=2.5cm]{geometry}

\makeatletter % makes "@" a normal character so that it can be used in macros
\def\@cite#1#2{{\m@th\upshape\bfseries%
[{#1\if@tempswa{\m@th\upshape\mdseries, #2}\fi}]}}
\makeatother % returns "@" to its normal state
\DeclareMathOperator{\GL}{GL}
\DeclareMathOperator{\SL}{SL}

\DeclareMathOperator{\Aut}{Aut}
\DeclareMathOperator{\Aff}{Aff}
\newcommand{\Z}{\mathbb{Z}}
\newcommand{\Q}{\mathbb{Q}}
\newcommand{\R}{\mathbb{R}}
\newcommand{\C}{\mathbb{C}}

\renewcommand*\cal[1]{\mathcal{#1}}

\newcommand*\define[1]{\textit{#1}}
\title{Periodic points on the regular and double $n$-gon surfaces}
\author{Paul Apisa, Rafael M.~Saavedra, and Christopher Zhang}
\date{\today }

\begin{document}

\theoremstyle{plain}
\newtheorem{theorem}{Theorem}[section]
\newtheorem{thm}[theorem]{Theorem}
\newtheorem{prop}[theorem]{Proposition}
\newtheorem{conj}[theorem]{Conjecture}
\newtheorem{lemma}[theorem]{Lemma}
\newtheorem{sublem}[theorem]{Sublemma}

\newtheorem{prob}[theorem]{Problem}
\newtheorem{axiom}[theorem]{Axiom}
\newtheorem{cor}[theorem]{Corollary}

\newtheorem{alg}[theorem]{Algorithm}
\newtheorem{form}[theorem]{Formula}

\theoremstyle{remark}
\newtheorem{rmk}[theorem]{Remark}
\newtheorem{ques}[theorem]{Question}
\newtheorem{motto}[theorem]{Motto}
\newtheorem{claim}[theorem]{Claim}

\theoremstyle{definition}
\newtheorem{defn}[theorem]{Definition}
\newtheorem{ex}[theorem]{Example}

\maketitle

\begin{abstract}
    Using the transfer principle, we classify the periodic points on the regular $n$-gon and double $n$-gon translation surfaces and deduce consequences for the finite blocking problem on rational triangles that unfold to these surfaces.
\end{abstract}

% \section{Outline for Paper}

% Intro

% (1) GL(2, R) acts on moduli space of translation surfaces

% (1.5) Remind reader what a Veech surface is

% (2) Define Periodic point, see Moller, Periodic points on Veech surfaces and the Mordell-Weil group over a Teichm\"uller curve

% (2.5) Periodic points are those that can be marked by a holomorphic section of the universal curve over a Teichmuller curve (Moller)

% (3) Main Theorem: The periodic points on the TC … are WPs

% (4) Unfolding - Katok-Zemlyakov

% (5) Corollary: Finite blocking

% Proof section:
% 1) Cite Veech to compute Veech group and note the number of cusps

% (2) Transfer

% (3) Rational height condition

% (4) Corollary using similar triangles

% (5) McMullen’s rationality computation

% (6) even- and odd-gon

%%%%%%%%%%%%%%%%%%%%%%%
%%%%%%%%%%%%%%%%%%%%%%%
%%%%%%%%%%%%%%%%%%%%%%%%%
%%%%%%%%%% Section 1 %%%%%%%%%%%%%%%%
%%%%%%%%%%%%%%%%%%%%%%%%
%%%%%%%%%%%%%%%%%%%%%

\section{Introduction}
The group $\GL^+(2,\R)$ acts on the moduli space of translation surfaces, which is stratified by specifying the number of singularities of the flat metric and their cone angles. This action, which is generated by complex scalar multiplication and Teichm\"uller geodesic flow, preserves these strata. In the sequel, an element of a stratum will be denoted $(X, \omega)$ where $X$ is a Riemann surface and $\omega$ is a holomorphic 1-form that induces the translation surface structure. Given such a point,  its stabilizer $\SL(X,\omega)$ in $\SL(2,\R)$ is called the \define{Veech group}, and $(X, \omega)$ is called a \define{Veech surface} when this group is a lattice.

\begin{defn}
A point $p$ on a Veech surface $(X, \omega)$ is called \emph{periodic} if $p$ is not a zero of $\omega$ and if its orbit under $\mathrm{Aff}(X, \omega)$--the affine diffeomorphism group of $(X, \omega)$--is finite. 
\end{defn}

\begin{rmk}
Our definition is equivalent to the one used in Apisa \cite{earle_kra}. A version of this definition which includes the zeros of $\omega$ first appeared in Gutkin-Hubert-Schmidt \cite{affine_diffeomorphisms}.  Under this definition, an equivalent notion of a periodic point is a point marked by a holomorphic multisection of the universal curve over a suitable finite cover of the Teichm\"uller curve associated to $(X, \omega)$. See M\"oller \cite[Lemma 1.2]{moeller_periodic_points} for details.
\end{rmk} 

Consider the following translation surfaces. For $n$ even, the \define{regular $n$-gon surface} is the regular $n$-gon with opposite sides identified, and for $n$ odd, the \define{double $n$-gon surface} is two copies of a regular $n$-gon differing by a rotation by $\pi$ with parallel sides glued together. The regular $10$-gon and the double $7$-gon surfaces are depicted in Figure~\ref{lines}. Both the regular and double $n$-gon surfaces are hyperelliptic with their hyperelliptic involutions being affine diffeomorphisms of derivative $-\mathrm{Id}$.  The \define{Weierstrass points} are the fixed points of this involution. In~\cite{veech_eisenstein}, Veech proved that the regular $n$-gon and double $n$-gon surfaces are Veech surfaces for all $n \geq 3$. Our main result is a classification of periodic points on these surfaces.

\begin{thm} \label{maintheorem}
When $n \geq 5$ and $n\ne 6$, the periodic points of the regular $n$-gon and double $n$-gon surfaces are exactly the Weierstrass points that are not singularities of the flat metric. 
\end{thm}

\begin{rmk}
When $n=5,8$, and $10$ this result was shown by M\"oller~\cite[Theorems 5.1, 5.2]{moeller_periodic_points}. When $n = 3, 4$ or $6$ the surfaces are tori and have infinitely many periodic points coming from torsion points.
\end{rmk}

The proof can be divided into two steps. First, we use the transfer principle to reduce the problem to classifying periodic points on an explicit set of saddle connections. Second, we classify the periodic points on these saddle connections by covering them with two collections of non-parallel cylinders and using the ``rational height lemma" of Apisa \cite{earle_kra}, which we will recall. 

The regular $n$-gon and double $n$-gon surfaces belong to a larger infinite family of Veech surfaces called the Veech-Ward-Bouw-M\"oller surfaces \cite{bouw_moeller}. It would be interesting to know whether our methods could be used to classify the periodic points on these surfaces, which, by Hooper \cite{grid_graphs} and Wright \cite{wright_bouw_moller}, also admit a presentation as a disjoint union of regular polygons with side identifications. 

%%%%%%%%

Theorem~\ref{maintheorem} has consequences for the finite blocking problem.

\begin{defn}
Two points $P,Q$ on a billiard table (resp.~translation surface) $M$ are \define{finitely blocked} if there is a finite set of points $S\subset M - \{P, Q\}$ such that all billiard trajectories (resp.~straight line segments that do not contain singularities in their interior) from $P$ to $Q$ pass through a point in $S$. 
\end{defn}

The following corollaries will be proven after the proof of the main theorem.

\begin{cor}\label{main_corollary}
When $n \geq 5$ and $n \ne 6$, the pairs of finitely blocked points on the regular $n$-gon and double $n$-gon surfaces consist precisely of any point that is not a singularity and its image under the hyperelliptic involution.
\end{cor}

Via the unfolding construction of Katok-Zemlyakov \cite{unfolding}, the $\left(\frac{\pi}{2}, \frac{\pi}{n}, \frac{(n-2)\pi}{2n} \right)$ triangle unfolds to the regular $n$-gon or double $n$-gon surface when $n$ is even or odd respectively. Therefore, a consequence of the previous corollary is the following.

\begin{cor} \label{application}
When $n \geq 5$ and $n \ne 6$, the $\left(\frac{\pi}{2}, \frac{\pi}{n}, \frac{(n-2)\pi}{2n}\right)$ triangle admits a pair of finitely blocked points if and only if $n$ is even, in which case the only such pair is the vertex of angle $\frac{\pi}{n}$ and itself. 
\end{cor}

\begin{rmk}
Similar statements can be deduced for the $\left(\frac{\pi}{n}, \frac{\pi}{n}, \frac{(n-2)\pi}{n} \right)$ and $\left(\frac{2\pi}{n}, \frac{(n-2)\pi}{2n}, \frac{(n-2)\pi}{2n} \right)$ triangles, which unfold to the regular $n$-gon surface, the double $n$-gon surface, or a double cover of one of those surfaces.
\end{rmk}

The remainder of the paper is divided into three sections. In Section \ref{S:Preliminaries} we establish some facts about the flat geometry of the regular and double $n$-gon surfaces. This will require a lemma on the rationality of ratios of sines that will be proven in Section \ref{S:Ratio of Sines}. Section \ref{S:MainTheorem} contains the proof of Theorem \ref{maintheorem} and its corollaries. 

\medskip \noindent \textbf{Context.}  Some of the earliest results on periodic points and the finite blocking problem, especially for Veech surfaces, are due to Gutkin-Hubert-Schmidt \cite{affine_diffeomorphisms}, Hubert-Schmoll-Troubetzkoy \cite{HST}, and Monteil \cite{Mont1, Mont2}; see especially \cite[Theorem 1]{Mont1} for related work on the finite blocking problem in the regular $n$-gon.

Periodic points in strata of Abelian and quadratic differentials were classified by Apisa \cite{earle_kra} and Apisa-Wright \cite{marked_points} respectively. Periodic points for every genus two translation surface were classified by M\"oller \cite{moeller_periodic_points} and Apisa \cite{Apisa:MarkedPointsGenus2}. These works  use the classification of $\mathrm{GL}(2, \R)$ orbit closures in genus two strata of Abelian differentials by McMullen \cite{Mc5}. The only periodic point on a genus two translation surface that is not a Weierstrass point was discovered by Kumar-Mukamel \cite{KM} and relates to orbit closures discovered by Eskin-McMullen-Mukamel-Wright \cite{EMMW}. 

For recent work on the finite blocking and illumination problems that was inspired by work of Eskin, Mirzakhani, and Mohammadi (\cite{EM, EMM}), see Leli\`evre-Monteil-Weiss \cite{LMW}, Apisa-Wright \cite{marked_points}, and Wolecki \cite{Wolecki}.

%\ann{P Oct 11: I added a section on context.}

% \begin{cor}
% There are no finitely blocked points on the $(\frac{\pi}{2}, \frac{\pi}{n}, \frac{(n-2)\pi}{2n})$ triangle. 
% \end{cor}

% \begin{proof}
% By \cite{marked_points}, a point can only be finitely blocked from itself on a $(\frac{\pi}{2}, \frac{\pi}{n}, \frac{(n-2)\pi}{2n})$ triangle. By theorem 3.15 in the same paper if $p$ is a point that is finitely blocked from itself, the blocking set can only contain $p$ and the points that lift to periodic points on the regular $n$-gon and double $n$-gon. 
% \end{proof}

%On the right angled triangle that unfolds, there is no blocking. On the isosceles triangle that unfolds, the only blocking is between the two vertices of equal angle and they are blocked by the midpoint. (Be careful about how the unfolding happens; for 4k, we get a double cover, otherwise the actual surface)

\subsection*{Acknowledgements} This work was done at the 2020 University of Michigan REU. We are grateful to Alex Wright as a co-organizer of the REU and supervisor of the project and for helpful conversations.
%, and to Paul Apisa as a co-organizer of the REU and supervisor and mentor of the project and for countless hours of helpful conversations, explanations, feedback, and advice. 
This material is based upon work supported by the National Science Foundation Grant No. DMS 1856155. During the preparation of this paper, the first author was partially supported by NSF Postdoctoral Fellowship DMS 1803625, and the third author was partially supported by the NSF Graduate Research Fellowship DGE 1841052.

%%%%%%%%%%%%%%%%%%%%%%%%%%%%%%%%%%%%%
%%%%%%%%%%%%%%%%%%%%%%%%%%%%%%%%
%%% SECTION 2 %%%%%%%%%
%%%%%%%%%%%%%%%%%%
%%%%%%%%%%%%%%%%%%%%%%%%%%

\section{Preliminaries}\label{S:Preliminaries}
% In~\cite[Definition 5.6, Theorem 5.8]{veech_eisenstein} (see also~\cite[Theorem~5.4]{mata}), Veech computed the Veech groups of the regular $n$-gon and double $n$-gon surfaces.

% Understanding the Veech group of a Veech surface is essential for determining its periodic points.

Fix an integer $n$ so that either $n = 5$ or $n \geq 7$. Let $R_1$ denote a regular $n$-gon circumscribed in the unit circle in $\C$ 
centered at the origin and so that one of its vertices lies at the point $i$. When $n$ is even, the regular $n$-gon surface is $R_1$ with opposite sides identified. To form the double $n$-gon surface when $n$ is odd we take a copy of $R_1$ rotated by $\frac{\pi}{n}$, which we call $R_2$, and identify parallel sides between $R_1$ and $R_2$. By triangulating these polygons and computing Euler characteristic, it is easy to see that the genus of the regular $n$-gon surface (resp. double $n$-gon surface) is $\lfloor \frac{n}{4}\rfloor$ (resp. $\frac{n-1}{2}$). 

%\ann{P: Sentence added. Corresponding sentence removed from intro and commented out here below.}  

% By counting Weierstrass points, we find the genus of a regular $n$-gon surface is $\lfloor \frac{n}{4}\rfloor $ and the genus of a double $n$-gon surface is $\frac{n-1}{2}$ \ann{P: I vote to move this sentence, which is a great addition, to the preliminary section}\ann{R: I WILL STATE WHAT THE WEIERSTRASS POINTS ARE IN THE PRELIMINARIES SECTION.}.

Let $\Gamma_n$ denote the Veech group of the regular $n$-gon surface when $n$ is even and the double $n$-gon surface when $n$ is odd. Make the following definitions,
\begin{align*}
    r_n = \begin{pmatrix}
    \cos \frac{\pi}{n} & -\sin \frac{\pi}{n} \\
    \sin \frac{\pi}{n} & \cos \frac{\pi}{n}
    \end{pmatrix}, && s_n = \begin{pmatrix}
    1 & 2\cot \frac{\pi}{n} \\
    0 & 1
    \end{pmatrix}.
\end{align*}

\begin{thm}[\cite{veech_eisenstein} (Definition 5.7, Theorem 5.8); see also~\cite{mata} (Theorem 5.4)] \label{veechgroup}
When $n$ is even, $\Gamma_n$ is generated by $\{r_n^2, s_n, r_ns_nr_n^{-1} \}$ and is isomorphic to the $(n/2,\infty ,\infty)$ triangle group. In particular, $\mathbb{H}/\Gamma_n$ has two cusps. 

When $n$ is odd, $\Gamma_n$ is generated by $\{r_n, s_n\}$ and is isomorphic to the $(2,n ,\infty)$ triangle group. In particular, $\mathbb{H}/\Gamma_n$ has one cusp.

% For $n$ even, the Veech group of the regular $n$-gon surface is generated by $\langle r_n^2, s_n, r_ns_nr_n^{-1} \rangle$ and is isomorphic to the $(n/2,\infty ,\infty)$ triangle group. For $n$ odd, the Veech group of the double $n$-gon surface is generated by $\langle r_n, s_n \rangle$ and isomorphic to the $(2,n ,\infty)$ triangle group. In particular these groups are lattices in $\SL(2,\R)$. In the former case the quotient of $\SL(2,\R)$ by the lattice has two cusps and in the latter case one cusp. 
\end{thm}

\begin{rmk}
In fact when $n$ is even, Veech considered a double cover of the regular $n$-gon surface; however, it is well-known that the Veech group of the two surfaces are identical.  Nevertheless, we will only ever use that, when $n$ is even, $\Gamma_n$ is contained in the Veech group of the regular $n$-gon surface, which is clear since each generator has that property. 
\end{rmk}

\begin{rmk}\label{R:CylinderCusp}
It is well known, see for instance Veech \cite[Section 3]{veech_eisenstein} or Hubert-Schmidt \cite[Lemma 4]{hubert_schmidt}, that for Veech surfaces, the maximal parabolic subgroups of the Veech group are in one-to-one correspondence with cylinder directions. The correspondence is given by associating to each cylinder direction, its stabilizer in the Veech group. Under this correspondence, the action of the Veech group on cylinder directions corresponds to its action by conjugation on maximal parabolic subgroups. Since conjugacy classes of maximal parabolic subgroups correspond to cusps of the quotient of the upper half plane by the Veech group, we see that each cusp corresponds to the orbit of a cylinder direction under the Veech group. Thus, every cylinder direction can be moved to one of these prescribed directions (described below) by an element of the Veech group.

In light of this observation, by Theorem \ref{veechgroup}, on the double $n$-gon surface  any cylinder direction may be sent to any other by an affine diffeomorphism. Similarly, there are two orbits of cylinders under the action of the affine diffeomorphism group on the regular $n$-gon surface. We will now describe these two cylinder directions. 

% Begin by scaling and rotating the regular $n$-gon (resp. double $n$-gon) surface so that (one of) the regular $n$-gon(s), which we will denote by $R$, that comprises it is circumscribed in a unit circle and so that it has a vertex lying directly above its midpoint.

The first is the horizontal direction, which is covered by $\left\lceil \frac{n}{4} \right\rceil$ (resp. $ \frac{n-1}{2} $) 
cylinders when $n$ is even (resp. odd), as seen on the left in Figure \ref{decagon}. Let $g_n'$ denote the number of horizontal cylinders. Notice that $g_n'$ is greater than or equal to the genus of the surface. Since the vertices of $R_1$ lie at the points $\left\{ i \mathrm{exp}\left( \frac{2j\pi i}{n} \right) \right\}_{j=0}^{n-1}$, it is easy to see that the heights of the horizontal cylinders are 
\[ h_j := \mathrm{Im}\left( i \mathrm{exp}\left( \frac{2j\pi i}{n} \right) - i \mathrm{exp}\left( \frac{(2j+2)\pi i}{n} \right) \right) \]
for $j \in \{0, \hdots, g_n' -1 \}$. We can simplify the expression for the heights as follows,
\begin{equation}\label{E:height1}
h_j = \mathrm{Im}\left( i \mathrm{exp}\left( \frac{(2j+1)\pi i}{n} \right)\left( \mathrm{exp}\left( \frac{-\pi i}{n} \right) - \mathrm{exp}\left( \frac{\pi i}{n} \right) \right)  \right) = 2 \sin\left( \frac{\pi}{n} \right) \sin \left( \frac{(2j+1)\pi}{n} \right). 
\end{equation} 

When $n$ is odd, notice that since $\sin(x) = \sin(\pi - x)$ for any $x$, for $j > \left\lfloor \frac{n-3}{4} \right\rfloor$ we can write $\sin \left( \frac{(2j+1)\pi}{n} \right)= \sin \left( \frac{(n-2j-1)\pi}{n} \right)$, and so
\begin{equation}\label{E:height2} \{ h_j \}_{j=0}^{g_n'-1} = \left\{ 2 \sin\left( \frac{\pi}{n} \right) \sin \left( \frac{k\pi}{n} \right) \right\}_{k=1}^{g_n'}. 
\end{equation}

\begin{figure}
    \centering
    \includegraphics[width=80mm]{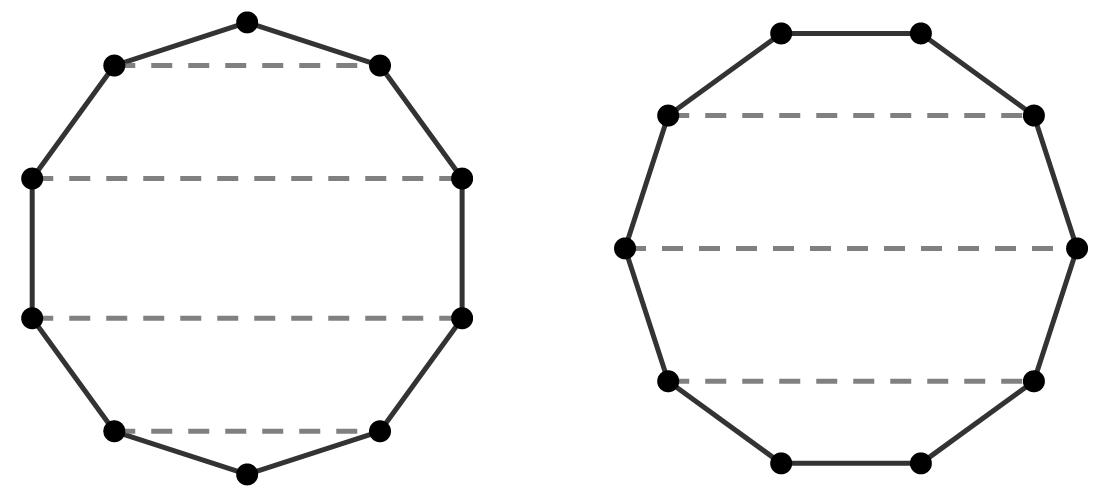}
    \caption{Two cylinder directions for the regular decagon.}
    \label{decagon}
\end{figure}

When $n$ is even, there are two orbits of cylinder directions under the action of the Veech group. Under the equivalences explained above, these two cylinder directions can be chosen to be stabilized by the maximal parabolic subgroup of $\Gamma_n$ generated by $s_n$ and $r_ns_nr_n^{-1}$. Using this observation (which is explained more fully in \cite[Section 5]{veech_eisenstein}), a  cylinder direction that cannot be sent to the horizontal one by the Veech group can be described as follows. Rotate the regular $n$-gon surface so that $R_1$ remains circumscribed in a unit circle but with one of its edges being horizontal. Let $R_1'$ denote this rotated copy of $R_1$. The horizontal direction is now covered by $g_n := \left\lfloor \frac{n}{4} \right\rfloor$ cylinders, as seen on the right in Figure~\ref{decagon}; the notation $g_n$ is chosen since it is equal to the genus of the surface. Since the vertices of $R_1'$ lie at the points $\left\{ i \mathrm{exp}\left( \frac{(2j+1)\pi i}{n} \right) \right\}_{j=0}^{n-1}$, it is easy to see that the heights of the horizontal cylinders are 
\[ h_j' := \mathrm{Im}\left( i \mathrm{exp}\left( \frac{(2j+1)\pi i}{n} \right) - i \mathrm{exp}\left( \frac{(2j+3)\pi i}{n} \right) \right) \]
for $j \in \{0, \hdots, g_n -1 \}$. We can simplify the expression for the heights as follows,
\begin{equation}\label{E:height3} h_j' = \mathrm{Im}\left( i \mathrm{exp}\left( \frac{(2j+2)\pi i}{n} \right)\left( \mathrm{exp}\left( \frac{-\pi i}{n} \right) - \mathrm{exp}\left( \frac{\pi i}{n} \right) \right)  \right) = 2 \sin\left( \frac{\pi}{n} \right) \sin \left( \frac{(2j+2)\pi}{n} \right).
\end{equation}
\end{rmk}

% The regular $n$-gon has two conjugacy classes of cylinder directions corresponding to the two conjugacy classes of maximal parabolic subgroups generated by $s$ and $rsr^{-1}$ in its Veech group. The double $n$-gon has a single conjugacy class of cylinder directions corresponding to the maximal parabolic subgroup generated by $s$.

% We will prove the following result on rationality of ratio of sines by following McMullen \cite[Pages 6-7]{ratio_sines}

%\ann{R: Nov 4: This was what I added. I commented out the previous version. \\ P: Rephrased again. Original commented out.}

The following result was stated in McMullen \cite[page 7]{ratio_sines} where it is indicated that its proof follows from an application of the bounds in the proof of \cite[Theorem 2.1]{ratio_sines}. Since the deduction is not entirely trivial, we offer a proof in Section \ref{S:Ratio of Sines}. In our deduction, we will not use the full strength of \cite[Theorem 2.1]{ratio_sines}, which shows that the number of sine ratios of any fixed degree over $\Q$ is finite and which can be used to find all such ratios.

% In \cite[Theorem 2.1]{ratio_sines}, McMullen shows that the number of sine ratios of any fixed degree over $\Q$ is finite. He also gives an explicit bound allowing one to effectively find them all. In Section~\ref{S:Ratio of Sines}, we prove the following result by means of those bounds.

%The following result was stated in McMullen \cite[page 7]{ratio_sines}, where it is indicated to be a consequence of \cite[Theorem 2.1]{ratio_sines}. Since the deduction is not entirely trivial, we offer a proof in Section \ref{S:Ratio of Sines}.

\begin{lemma}\label{L:RatioSines}
For rational numbers $0 < \alpha < \beta \leq \frac{1}{2}$, $\frac{\sin(\pi \alpha)}{\sin(\pi \beta)}$ is rational if and only if $\alpha = \frac{1}{6}$ and $\beta = \frac{1}{2}$.
\end{lemma}

\begin{lemma}\label{heights}
On the regular $n$-gon and double $n$-gon surfaces, at least one cylinder direction has the property that the ratio of heights and circumferences of distinct cylinders in this direction have irrational ratio; every cylinder direction has this property whenever $n$ is not congruent to $0$ or $6$ mod $12$. Moreover, for any two parallel cylinders sharing a boundary saddle connection, their heights and circumferences have an irrational ratio. 
\end{lemma}
\begin{proof}
By Remark \ref{R:CylinderCusp}, any cylinder direction can be sent to one of the two directions specified in Remark \ref{R:CylinderCusp} by an element of the Veech group. In particular, the ratio of heights of two distinct parallel cylinders is either $\frac{h_j}{h_k}$ or $\frac{h_j'}{h_k'}$. By Equations (\ref{E:height1}), (\ref{E:height2}), and (\ref{E:height3}), up to inverting the ratio, these ratios always have the form $\frac{\sin(\pi \alpha)}{\sin(\pi \beta)}$ for rational numbers $0 < \alpha < \beta \leq \frac{1}{2}$ where $n\alpha$ and $n\beta$ are integers. By Lemma \ref{L:RatioSines} such a ratio is rational if and only if $\alpha = \frac{1}{6}$ and $\beta = \frac{1}{2}$. In particular by Equation (\ref{E:height2}) the ratio of heights of distinct parallel cylinders is irrational when $n$ is odd.

Suppose therefore that $n$ is even. By Equation (\ref{E:height1}), $\frac{h_j}{h_k} = \frac{\sin\left( \frac{(2j+1)\pi}{n} \right)}{\sin\left( \frac{(2k+1)\pi}{n} \right)}.$
If $j < k$ then this ratio is rational if and only if $\frac{2j+1}{n} = \frac{1}{6}$ and $\frac{2k+1}{n} = \frac{1}{2}$. This implies that $n = 12j + 6$ and $k = 3j+1$. Therefore, the ratio of heights of distinct cylinders in this cylinder direction is irrational if and only if $n$ is not congruent to $6$ mod $12$. Moreover, since $j$ and $k$ correspond to cylinders that share a boundary saddle connection if and only if $|j-k|=1$ we have that cylinders that share a boundary saddle connection have an irrational ratio of height as long as $j > 0$, which is the case since we have assumed that $n \ne 6$.

By Equation (\ref{E:height3}), $\frac{h_j'}{h_k'} = \frac{\sin\left( \frac{(2j+2)\pi}{n} \right)}{\sin\left( \frac{(2k+2)\pi}{n} \right)}.$
If $j < k$ then this ratio is rational if and only if $\frac{2j+2}{n} = \frac{1}{6}$ and $\frac{2k+2}{n} = \frac{1}{2}$. This implies that $n = 12j + 12$ and $k = 3j+2$. Therefore, the ratio of heights of distinct cylinders in this cylinder direction is irrational if and only if $n$ is not congruent to $0$ mod $12$. As before cylinders that share a boundary saddle connection have an irrational ratio of height.

By Remark \ref{R:CylinderCusp} any cylinder direction can be moved by an element of the Veech group to one of the two cylinder directions analyzed in the preceding paragraphs, so the result follows. The claims for circumferences hold since the ratio of moduli of parallel cylinders is rational for any Veech surface. 
%\ann{P: Citation to Veech's paper removed}  
%\cite[Equation 5.2]{veech_eisenstein}. 
%\ann{C: added citation \\ P: Is the citation for the $n$-gon and double $n$-gon surfaces or for a general Veech surface? Since this fact is standard perhaps we could put the citation in a parenthetical?} 
\end{proof}

% By Remark \ref{R:CylinderCusp} it suffices to prove this statement for the one (resp. two) cylinder direction(s) identified in the remark. As shown in the remark the ratio of heights of distinct parallel cylinders always takes the form $\frac{\sin(\pi \alpha)}{\sin(\pi \beta)}$ for rational numbers $0 < \alpha < \beta \leq \frac{1}{2}$ where $n\alpha$ and $n\beta$ are integers. By \cite[Page 7]{ratio_sines} this ratio is rational if and only if $\alpha = \frac{1}{6}$ and $\beta = \frac{1}{2}$ \ann{P: [IMPORTANT!] Make sure this claim is true and find a better citation}. In particular the ratio is irrational when $n$ is odd.  When $n = 12k+6$ for any nonnegative integer $k$ we see that $\frac{h_{3k+1}}{h_k}$ is rational and no other ratio of heights of distinct parallel cylinders listed in Remark \ref{R:CylinderCusp} is. Similarly, when $n = 12k$ for any positive integer $k$ we see that $\frac{h'_{3k}}{h'_{2k-1}}$ is rational and no other ratio of heights of distinct parallel cylinders listed in Remark \ref{R:CylinderCusp} is \ann{P: [IMPORTANT!] Consider adding more detail; definitely need to explain the adjacency comment and more generally how the proof shows what is claimed}. The claim for heights follows from the observation that the ratio of moduli of parallel cylinders is rational for any Veech surface. 

% Note that there are simpler proofs of Corollary \ref{affine} below than by using Lemma \ref{min_cover}, but Lemma 2.2 will be used in the proof of Corollary \ref{main_corollary}.

\begin{defn}
Given translation surfaces $(X, \omega)$ and $(X', \omega')$ a \emph{translation cover} $f: (X, \omega) \rightarrow (X', \omega')$ is a holomorphic map $f: X \rightarrow X'$ such that $f^* \omega' = \omega$. Similarly, if $(Y, q)$ is a quadratic differential that is not the square of a holomorphic one-form, then we say that $f: (X, \omega) \rightarrow (Y, q)$ is a \define{half-translation cover} if $f: X \rightarrow Y$ is holomorphic and $f^*q = \omega^2$.
\end{defn}

\begin{lemma}\label{min_cover}
Let $(X,\omega)$ be the regular $n$-gon or double $n$-gon surface. If $(X', \omega')$ is a translation surface so that $f: (X, \omega) \rightarrow (X', \omega')$ is a translation cover then $(X, \omega) = (X', \omega')$.
% There is no lower genus Riemann surface $X'$ so that $\omega$ is the pullback of a holomorphic one-form on $X'$ under a holomorphic map from $X$ to $X'$ \ann{P: Rephrased since translation cover is not standard terminology}.
\end{lemma}
\begin{proof}
Suppose to a contradiction that there is a translation cover $f: (X, \omega) \rightarrow (X', \omega')$ where the genus $g'$ of $X'$ is less than the genus $g$ of $X$. For each cylinder $C$ of circumference $c$ on $(X, \omega)$, there is an integer $m$ so that $f(C)$ is a cylinder of circumference $c/m$. By Lemma \ref{heights}, there is a cylinder direction on $(X, \omega)$ in which all distinct pairs of cylinders have an irrational ratio of circumferences and hence map to distinct cylinders on $(X', \omega')$ under $f$. Since every cylinder direction on $(X, \omega)$ has at least $g$ cylinders (see the description of cylinder directions in Remark \ref{R:CylinderCusp}), it follows that there is a cylinder direction on $(X', \omega')$ with $g$ distinct cylinders. Since $\omega$ has at most two zeros, the number of zeros $s'$ of $\omega'$ is also at most two. Therefore, $g \leq g' + 1$, since the number of cylinders on $(X', \omega')$ is bounded above by $g' + s' - 1$. Since we have assumed that $g' < g$ we see that $g' = g-1$ and $s' = 2$. 

By the Riemann-Hurwitz formula, this implies that $g \leq 3$ and hence that $n \in \{5,7,8,10,12,14\}$. The condition that $\omega$ has two singularities reduces the possibilities to just $n \in \{10, 14\}$. Since $10$ and $14$ are not congruent to $0$ or $6$ mod $12$ it follows from Lemma \ref{heights} that these surfaces possess cylinder directions with $g+1$ cylinders so that the ratio of circumferences of distinct cylinders is irrational. As argued above, this implies that $(X', \omega')$ also has such a cylinder direction and hence that $g = g'$, which is a contradiction. 
\end{proof}

% any two cylinders in the same direction have an irrational ratio of circumferences, must be sent to distinct cylinders in $(X',\omega')$. 

% For $n=4k+2$, the regular $n$-gon $M$ belongs to $\cal H(k-1,k-1)$ which is a locus of genus $k$ translation surfaces, and there is a cylinder direction on $(X,\omega)$ with $k+1$ cylinders. Thus, $(X',\omega')$ has $k+1$ cylinders in that direction and at most $2$ singularities.  

% The shears of these $k+1$ cylinders span a $(k+1)$-dimensional subspace in $H^1(X,\Sigma;\C)$. Since there are at most $2$ singularities, the kernel of the projection $H^1(X,\Sigma;\C)\to H^1(X;\C)$ is dimension at most $1$, so the space of shears of the cylinders projects to an isotropic subspace of absolute cohomology of dimension at least $k$. Thus, $(X',\omega')$ has genus at least $k$, so $\pi_\omega$ must be degree 1. A similar argument holds for $n=4k$ and $n=2k+1$ since $(X,\omega)$ belongs to $\cal H(2k-2)$, and there exists a cylinder direction with $k$ cylinders. 

\begin{cor}\label{affine}
For the regular $n$-gon and double $n$-gon, the affine diffeomorphism group is isomorphic to the Veech group.
\end{cor}
\begin{proof}
Let $(X,\omega)$ be the regular $n$-gon or double $n$-gon surface. Letting $\Aut(X,\omega)$ be the group of affine diffeomorphisms of derivative $\mathrm{Id}$, we have the following short exact sequence (see for instance \cite[Equation (2.6)]{veech_eisenstein}).
%There is an exact sequence
\[
\begin{tikzcd}
0 \arrow[r] & {\Aut(X, \omega)} \arrow[r] & {\Aff(X, \omega)} \arrow[r] & {\SL(X, \omega)} \arrow[r] & 0.
\end{tikzcd}
\]
It suffices to show that $\Aut(X,\omega)$ is trivial. This follows from Lemma \ref{min_cover} since the cover $(X,\omega)\to (X,\omega)/\Aut(X,\omega)$ must be the identity.
%Consider all horizontal saddle connections. These saddle connections come in pairs of equal length except for the shortest such connection. A translation surface automorphism takes a saddle connection to another of one of the same holonomy, so it must fix the shortest horizontal saddle connection. A translation surface automorphism that fixes a nonsingular point must be the identity. 
\end{proof}

\begin{rmk}
It is standard that the Weierstrass points that do not coincide with singularities of the flat metric are periodic points for any translation surface in a hyperelliptic locus of a stratum. However, in our case, this is particularly easy to see since the isomorphism between $\mathrm{Aff}(X, \omega)$ and $\mathrm{SL}(X, \omega)$ shows that the hyperelliptic involution is in the center of $\mathrm{Aff}(X, \omega)$--since it is sent to $-\mathrm{Id}$--and hence that $\mathrm{Aff}(X, \omega)$ permutes its fixed points. 
\end{rmk}

\begin{rmk}\label{R:FixedPoint}
Let $p$ provisionally denote the center of the regular $n$-gon $R_1$ whose opposite sides are identified to form the regular $n$-gon surface. We will show that $p$ is fixed by every element of the affine diffeomorphism group. It is obvious that it is fixed by the rotation $r_n$. The remaining generators of the affine diffeomorphism group are shears in the cylinder directions identified in Remark \ref{R:CylinderCusp}. In cylinder direction of slope $\frac{\pi}{n}$, $p$ lies on a boundary saddle connection of a cylinder and is trivially fixed. In the horizontal cylinder direction, $p$ lies on the core curve of a cylinder $C$ of modulus $\tan(\pi/n)$. The corresponding generator of the maximal parabolic subgroup is $s_n$, which performs two Dehn twists in $C$ and hence fixes $p$. Since $p$ is fixed by the generators of the affine diffeomorphism group, it is fixed by every element in it.
\end{rmk}

%%%%%%%%%%%%%%%%%%%%%%%%%%%%%%%%%
%
% SECTION 3
%
%%%%%%%%%%%%%%%%%%%%%%%%%%%%%%%%%%%%

\section{Proof of Theorem \ref{maintheorem} and its corollaries}\label{S:MainTheorem}

We now begin our study of periodic points using the transfer principle. Our goal is to reduce the main theorem to identifying the periodic points on finitely many saddle connections. We start with the following definition.

\begin{defn}
When $n$ is even let $P_n$ denote the point in Remark \ref{R:FixedPoint}. When $n$ is odd, let $P_n$ denote the unique cone point of the flat metric of the double $n$-gon surface. Both points are fixed by every element of $\Gamma_n$.
\end{defn}

% By Corollary~\ref{affine}, the hyperelliptic involution, which corresponds to $-\Id\in \SL(X,\omega)$, is in the center of the affine diffeomorphism group. Thus, the Weierstrass points are periodic points. The difficulty in proving Theorem~\ref{maintheorem} lies in showing that all other points are not periodic. We first use the transfer principle to rule out all points except finitely many lines.

\begin{prop} \label{transfer}
The $\Gamma_n$ orbit of any periodic point on the regular $n$-gon or double $n$-gon surface contains a point lying on the leaf of the horizontal foliation passing through $P_n$ or, when $n$ is even, a point lying on the leaf of the foliation passing through $P_n$ that makes an angle of $\frac{\pi}{n}$ with the horizontal (see Figure \ref{lines}).

% Any periodic point on the regular $n$-gon surface lies on a saddle connection or cylinder core curve passing through $P_n$ 

% Recall that the regular $n$-gon surface is formed by identifying opposite sides on a regular $n$-gon $R$. If $L$ denotes the union of a line segment joining the midpoint of $R$ to a vertex with a line segment joining the midpoint of $R$ to the midpoint of an edge, then the orbit of any periodic point under the Veech group action intersects $L$ 
% Let $\Gamma$ be the Veech group of the regular $n$-gon or double $n$-gon $(X,\omega)$. If $n$ is even, all periodic points must lie on the $\Gamma$ orbit of one of the following two segments: the segment connecting the center of the polygon to a vertex of the $n$-gon and the segment connecting the center to the midpoint of an edge, as shown in Figure \ref{lines} left. If $n$ is odd, all periodic points must lie on the $\Gamma$ orbit of a horizontal saddle connection shown in Figure \ref{lines} right.
\end{prop}

\begin{figure}[h]
    \centering
    \includegraphics[width=100mm]{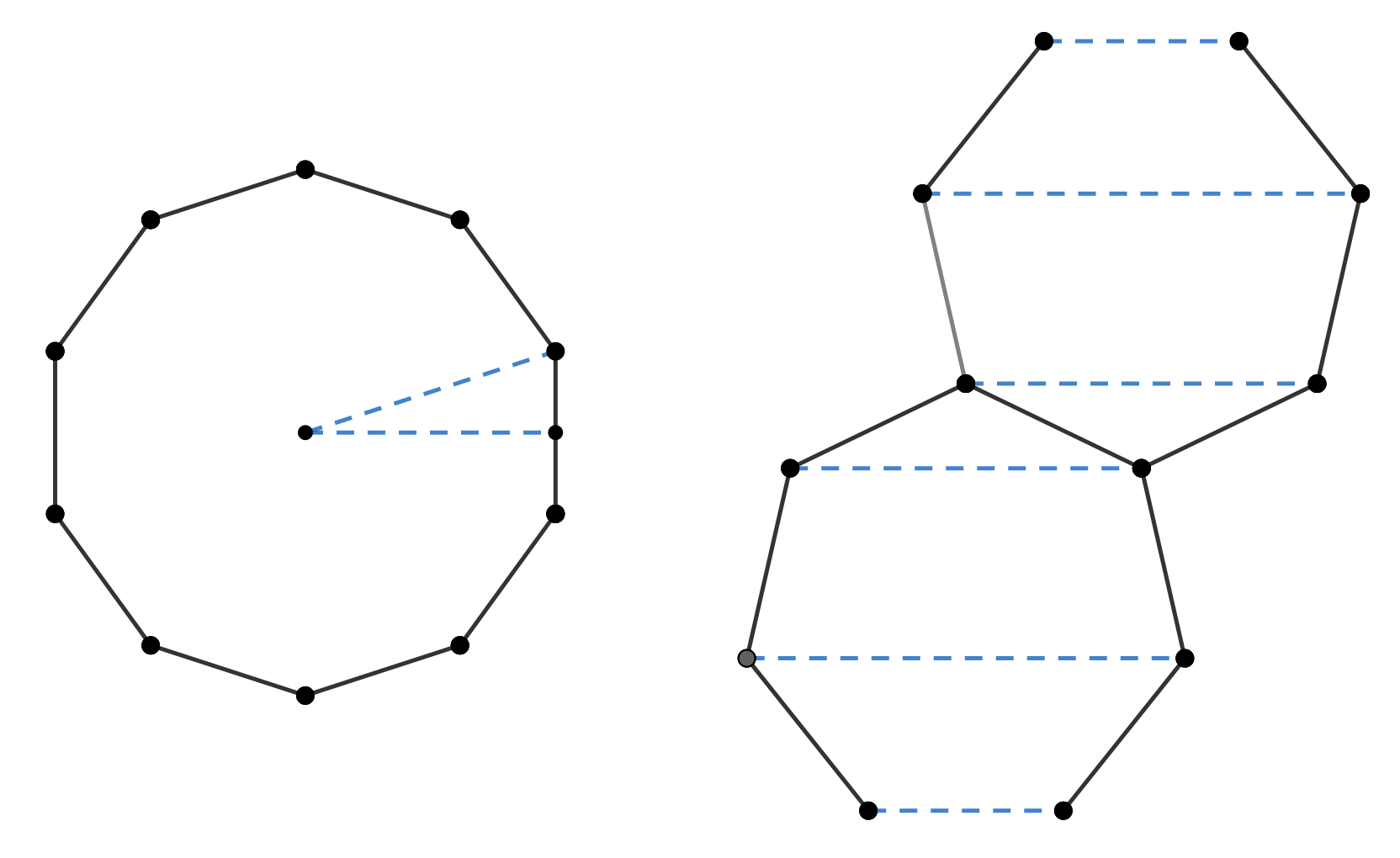}
    \caption{After applying Proposition \ref{transfer}, any periodic point can be assumed to lie on one of the dashed lines or its image under the hyperelliptic involution.}
    \label{lines}
\end{figure}

\begin{proof}
The transfer principle states that if $G$ and $H$ are topological groups acting continuously, from the left and right respectively, on a topological space $\cal X$ then the following are in bijective correspondence
\begin{enumerate}
    \item Closed (resp. dense) $G$ orbits on $\mathcal{X}/H$.
    \item Closed (resp. dense) $G \times H$ orbits on $\mathcal{X}$.
    \item Closed (resp. dense) $H$ orbits on $G \backslash \mathcal{X}$
\end{enumerate}
Under these correspondences, a $G\times H$ orbit of $x\in \cal X$ will be sent to the $G$ orbit of $xH$ or the $H$ orbit of $Gx$.
In our context, $G$ is the Veech group $\Gamma_n$, $\cal X$ is $\SL(2,\R)$, and $H$ is the unipotent subgroup $U:=\left\{\begin{pmatrix}
1 & s\\
0 & 1
\end{pmatrix}: s\in \R
\right\}$.
%Notice that the $\Gamma_n$ action on $\mathrm{SL}(2, \R)/U$ is given by the standard linear action of $\mathrm{SL}(2, \R)$ on $\R^2 - \{0\}$. The identification is given by sending an element $g \in \mathrm{SL}(2, \R)$ to $g \cdot  \begin{pmatrix} 1\\ 0 \end{pmatrix}$. 
$\SL(2,\R)/U$ can be identified with $\R^2-\{0\}$ by sending $g\in \SL(2,\R)$ to $g\cdot
\begin{pmatrix}
1\\
0
\end{pmatrix}
$. Under this identification, the action of $\Gamma_n$ on $\SL(2,\R)/U$ is given by the standard linear action of $\SL(2,\R)$ on $\R^2-\{0\}$.
Define $a_t := \begin{pmatrix} e^t & 0 \\ 0 & e^{-t} \end{pmatrix}$ where $t$ is any real number.

It is a foundational result of Dani \cite[Theorem A]{Dani} that the only $U$ orbits of $\Gamma_n \backslash \SL(2,\R)$ are closed or dense, and the closed orbits are horocycles around the cusps. Recall from Theorem \ref{veechgroup} that cusps of $\Gamma_n \backslash \SL(2,\R)$  correspond to conjugacy classes of maximal parabolic subgroups, and these are generated by $s_n$ and, when $n$ is even, $r_n s_n r_n^{-1}$. In particular, the closed horocycles corresponding to the cusps of $\Gamma_n \backslash \mathrm{SL}(2, \R)$ are given by $\Gamma_n a_t U$ and also, when $n$ is even, $\Gamma_n r_n a_t U$ where $t$ is any real number. By the transfer principle, the only vectors in $\R^2 - \{0\}$ that do not have dense $\Gamma_n$ orbit are vectors parallel to a vector in $\Gamma_n \cdot \begin{pmatrix} 1 \\ 0 \end{pmatrix}$ or, when $n$ is even, parallel to a vector in $\Gamma_n \cdot \begin{pmatrix} \cos \frac{\pi}{n} \\ \sin \frac{\pi}{n} \end{pmatrix}$.

Now let $p$ be any periodic point that is distinct from $P_n$. By definition, the orbit of $p$ under $\Gamma_n$ is finite. In particular, $\Gamma_n \cdot p$ remains a bounded distance away from $P_n$. Since the regular $n$-gon surface is comprised of a convex polygon with opposite sides identified and since the double $n$-gon surface is comprised of two regular $n$-gons whose vertices correspond to $P_n$, it follows that there is straight line $\gamma$ from $P_n$ to $p$, the holonomy of which we will denote by $v$. Since $p$ remains a bounded distance away from $P_n$ we have that $\Gamma_n \cdot v$ is not dense in $\R^2$ (in particular, the orbit is not dense in a neighborhood of $0$) and hence there is a vector in the $\Gamma_n$ orbit of $v$ that is parallel to a vector in $\Gamma_n \cdot \begin{pmatrix} 1 \\ 0 \end{pmatrix}$ or, when $n$ is even, one in $\Gamma_n \cdot \begin{pmatrix} \cos \frac{\pi}{n} \\ \sin \frac{\pi}{n} \end{pmatrix}$.  This shows that $\Gamma_n \cdot p$ contains a point on either the horizontal leaf through $P_n$ or, in the case of $n$ even, the leaf that makes an angle of $\frac{\pi}{n}$ with the horizontal.

% This implies that $p$ lies on the boundary saddle connection of a cylinder when $n$ is odd and that it lies on a saddle connection or cylinder core curve passing through $P_n$ when $n$ is even.
\end{proof}

\begin{defn}\label{D:CandidateLineSegment}
When $n$ is odd, call the saddle connections identified in Proposition \ref{transfer} \emph{candidate line segments}. When $n$ is even, notice that the hyperelliptic involution fixes both lines identified in Proposition \ref{transfer}. Each line can be partitioned into two subsegments, which are exchanged by the hyperelliptic involution and have one endpoint at $P_n$. For each line identified in Proposition \ref{transfer} we choose one of these subsegments and call them \emph{candidate line segments} (see Figure \ref{lines}). Since the hyperelliptic involution preserves the collection of periodic points, any periodic point can be moved by an element of the Veech group to a candidate line segment. %\ann{P: Definition updated. Original commented out.}
\end{defn}

In the sequel, we will adopt the convention that all cylinders are assumed to be closed; that is, they contain their boundary.

\begin{defn}
A point contained in a cylinder $C$ is said to have \emph{rational height in $C$} if its distance from the boundary of $C$ is a rational multiple of the height of $C$. 
\end{defn}

A more general form of the following lemma, called the ``rational height lemma", appeared in Apisa \cite[Lemma 5.4]{earle_kra}.

\begin{lemma} \label{rationalheightlemma}
A periodic point on a Veech surface has rational height in any cylinder containing it.
\end{lemma}
\begin{proof}
Let $p$ be the periodic point and suppose that it is contained in a cylinder $C$. After perhaps rotating the surface we may suppose without loss of generality that $C$ is horizontal. Denote its height and circumference by $h$ and $c$ respectively. Since the surface is Veech, the Veech group contains an element $g:=\begin{pmatrix} 1 & s \\ 0 & 1 \end{pmatrix}$ where $s = kc/h$ for some nonzero integer $k$. Choosing flat coordinates so that the bottom boundary of $C$ lies on the $x$-axis, we see that $g$ sends a point $(x, y) \in C$ to $(x+ykc/h,y)$ where the $x$ coordinate is taken modulo $c$. Thus, $(x,y)$ has finite orbit if and only if $y$ is a rational multiple of $h$.
\end{proof}

In the following lemma, it will be useful to use the notation $\overline{PQ}$ to refer to a straight line on a flat surface between points $P$ and $Q$. In general on a flat surface there are infinitely many straight lines between any two points, so we emphasize that this notation presupposes a choice of a line between $P$ and $Q$ and that the line is not uniquely determined by its endpoints. %\ann{P: warning added}

\begin{lemma} \label{cylinderlemma}
Let $C_2$ and $C_3$ be two parallel cylinders whose ratio of heights is irrational. Let $C_1$ be another cylinder. Suppose that $\overline{PQ}$ is a line segment satisfying the following:
\begin{enumerate}
    \item\label{I:oblique} $\overline{PQ}$ is neither parallel nor perpendicular to the core curves of $C_1$, $C_2$, and $C_3$.
    \item\label{I:InC1} $\overline{PQ}$ is contained in $C_1$ and its interior does not intersect the boundary of $C_1$. 
    \item\label{I:InC2} $\overline{PQ}$ is contained in $C_2 \cup C_3$ and its interior intersects the boundary of $C_2$ and $C_3$ in a unique point $R$. $\overline{PR}$ (resp. $\overline{RQ}$) is contained in $C_2$ (resp. $C_3$) and has nonzero length, see Figure \ref{cylinderlemmafig}.
    \item\label{I:NoWrapping} The orthogonal projection of $\overline{PQ}$ (resp. $\overline{PR}$, $\overline{RQ}$) to the core curve of $C_1$ (resp. $C_2$, $C_3$) is a proper subset of the core curve. 
    %(heuristically this means that these lines do not twist around the cylinders containing them). 
    \item\label{I:RationalHeight} $P$ (resp. $Q$) has rational height in $C_1$ and $C_2$ (resp. $C_1$ and $C_3$).
\end{enumerate}
Then the only point on $\overline{PR}$ (resp. $\overline{RQ})$ that has rational height in both $C_1$ and $C_2$ (resp. $C_1$ and $C_3$) is $P$ (resp. $Q$). %\ann{P: Rephrased and corrected the lemma. The proof was rewritten.} 
\end{lemma}

\begin{figure}[h]
    \centering
    \includegraphics[width=90mm]{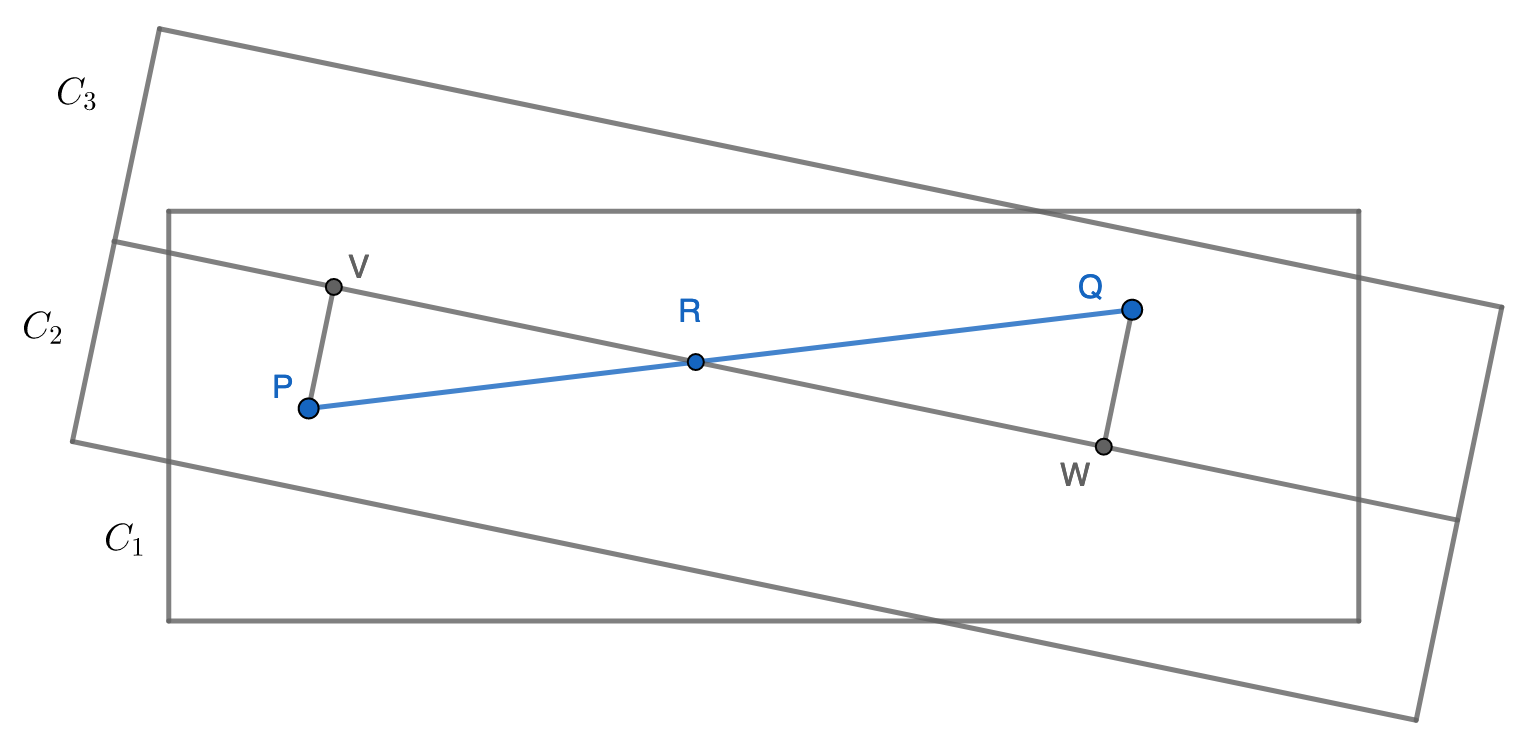}
    \caption{The three cylinders in Lemma~\ref{cylinderlemma}.}
    \label{cylinderlemmafig}
\end{figure}

\begin{proof}
Assume for the sake of contradiction that there is a point $S$ on $\overline{PR}$ other than $P$ with rational height inside both $C_1$ and $C_2$. Without loss of generality, perhaps after rotating the surface, we may suppose that $C_1$ is horizontal. 

When $P$ is contained in the interior of $C_1$, let $\ell$ denote the leaf of the horizontal foliation passing through $P$. When $P$ is contained in the boundary of $C_1$, let $\ell$ denote the boundary of $C_1$ containing $P$. Since the interior of $\overline{PQ}$ is contained in the interior of $C_1$ (by \eqref{I:InC1}), we may think of $C_1$ as a Euclidean cylinder and orthogonally project $\overline{PQ}$ onto $\ell$. By \eqref{I:oblique} and \eqref{I:NoWrapping}, this projection is a line $\overline{PU}$ where $U \ne P$. Let $\overline{QU}$ denote the vertical line contained in $C_1$ from $Q$ to $U$. The triangle, which we will denote $\Delta PQU$, formed by $\overline{PQ}$, $\overline{PU}$, and $\overline{QU}$ is, by \eqref{I:InC1}, a right triangle contained in $C_1$ as shown in Figure~\ref{trianglefig}. Let $T$ be the point on $\overline{QU}$ so that $\Delta SQT$ is similar to $\Delta PQU$ (see Figure \ref{trianglefig}). 

% with hypotenuse $\overline{PQ}$ and sides perpendicular and parallel to the boundary saddle connections of $C_1$. 

\begin{figure}[h]
    \centering
    \includegraphics[width=100mm]{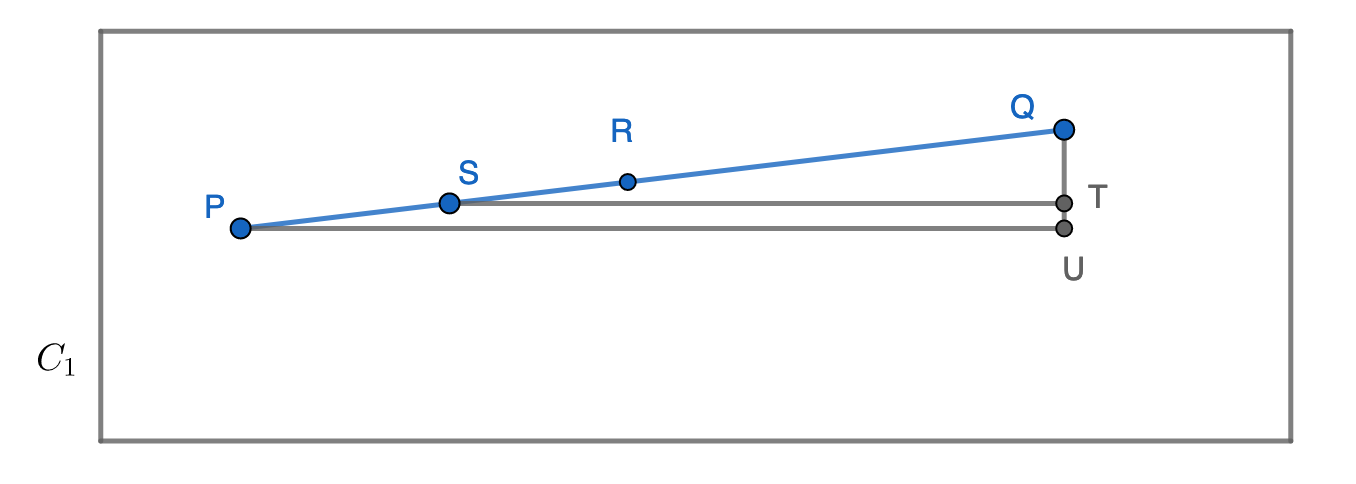}
    \caption{The triangles $SQT$ and $PQU$ are similar.}
    \label{trianglefig}
\end{figure}

Since, by \eqref{I:RationalHeight}, $P$, $Q$, and $S$ (and hence also $T$ and $U$) have rational height in $C_1$, it follows that the length of $\overline{UT}$ is a rational multiple of the length of $\overline{TQ}$. Since $\Delta SQT$ and $\Delta PQU$ are similar, it follows that the length of $\overline{PS}$ is a rational multiple of the length of $\overline{SQ}$.

The preceding argument (with the role of $\overline{PQ}$ being played by $\overline{PR}$ and that of $C_1$ by $C_2$) shows that the length of $\overline{PS}$ is a rational multiple of the length of $\overline{SR}$ (notice that \eqref{I:InC2} is the analogue of \eqref{I:InC1} here). It follow that the lengths of $\overline{PR}$ and $\overline{RQ}$ have a rational ratio. 

Finally, let $V$ (resp. $W$) denote the orthogonal projection of $P$ (resp. $Q$) to the component of the boundary of $C_2$ (resp. $C_3$) containing $R$ (see Figure~\ref{cylinderlemmafig}). We note that the triangles $\Delta PVR$ and $\Delta QWR$, which are formed in the same way we formed $\Delta PQU$, are similar. Since the lengths of $\overline{PR}$ and $\overline{RQ}$ have a rational ratio, so do the lengths of $\overline{PV}$ and $\overline{RW}$. However, by \eqref{I:RationalHeight}, the lengths of $\overline{PV}$ and $\overline{RW}$ are rational multiples of the height of $C_2$ and $C_3$ respectively. Therefore, we have shown that $C_2$ and $C_3$ have a rational ratio of heights, which is a contradiction.

\end{proof}

%\ann{P: Completely rewrote the proof}

\begin{proof}[Proof of Theorem~\ref{maintheorem}]

Suppose first that $n$ is even. By Proposition \ref{transfer}, any periodic point can be moved by an element of the Veech group to one of the two candidate line segments (see Definition \ref{D:CandidateLineSegment} and Figure \ref{evencases}). The endpoints of the candidate line segments are singularities of the flat metric or Weierstrass points. It suffices to show that these endpoints are the only periodic points on a candidate line segment. 
Choosing a candidate line segment, let $P$ denote $P_n$, which is one endpoint, and let $Q$ denote the other endpoint. We will let $\overline{PQ}$ denote the candidate line segment. Notice that $\overline{PQ}$ is contained in a single cylinder $C_1$ that makes an angle of $-\frac{\pi}{n}$ with the horizontal and to which $\overline{PQ}$ is not parallel. This is the dotted cylinder in Figure \ref{evencases}. 

The line $\overline{PQ}$ is also contained in the union of two parallel cylinders $C_2$ and $C_3$ in the cylinder direction that makes an angle of $-\frac{2\pi}{n}$ with the horizontal. The cylinders $C_2$ and $C_3$ share a boundary saddle connection and so by Lemma \ref{heights} they have an irrational ratio of heights. These cylinders are the dashed cylinders in Figure \ref{evencases}.
 
 \begin{figure}[h]
    \centering
    \includegraphics[width=50mm]{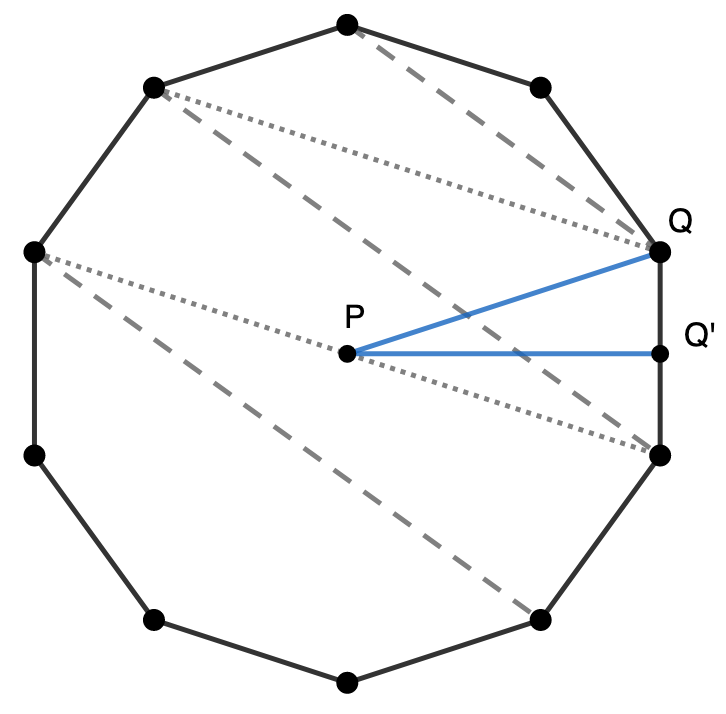}
    \caption{The cylinders dividing candidate lines segments in the even case. Both $Q$ and $Q'$ are endpoints of candidate line segments.}
    \label{evencases}
\end{figure}

Since the endpoints of the candidate line segments are either singularities of the flat metric or periodic points they have rational height in any cylinder containing them by Lemma \ref{rationalheightlemma}. By Lemma \ref{cylinderlemma} it follows that any point lying in the interior of $\overline{PQ}$ has irrational height in at least one of $C_1$, $C_2$, and $C_3$. Therefore, any point lying in the interior of a candidate line segment cannot be periodic by Lemma \ref{rationalheightlemma} as desired. 

Now suppose that $n$ is odd. Recall that the double $n$-gon surface is comprised of two regular $n$-gons, which we denoted by $R_1$ and $R_2$, that differ from each other by a rotation of $\pi/n$ and so that parallel sides are identified. Recall too that the hyperelliptic involution exchanges $R_1$ and $R_2$. Since every candidate line segment is contained in either $R_1$ or $R_2$ it suffices to classify the periodic points on the candidate lines in just $R_2$. Recall that we have supposed that $R_2$ is circumscribed in a circle centered at the origin in $\C$ and has a vertex lying at the point $-i$.

\begin{figure}[h]
    \centering
    \includegraphics[width=110mm]{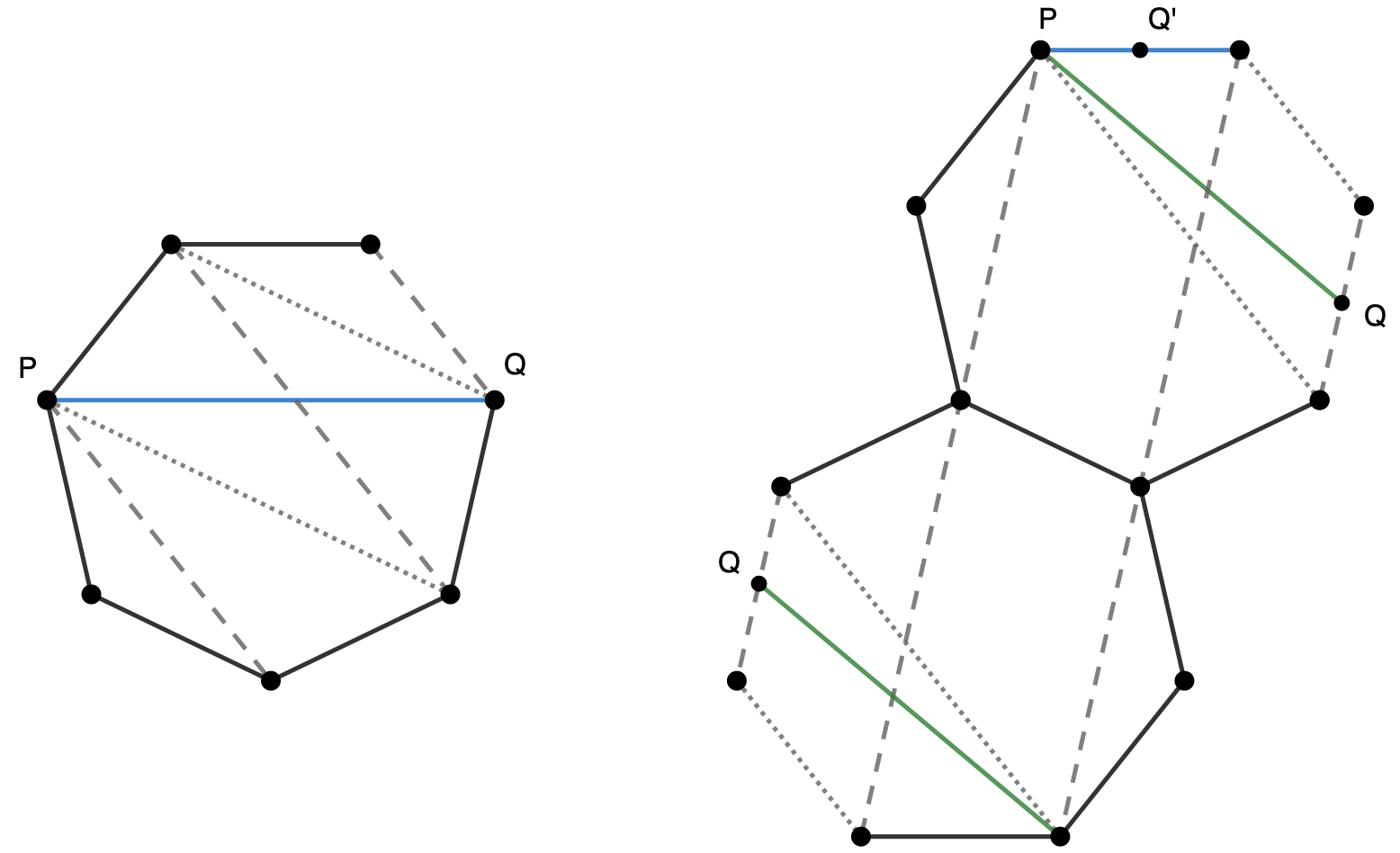}
    \caption{The cylinders dividing candidate lines in the odd case. The boundary of $C_1$ (resp. $C_2$ and $C_3$) are the dotted (resp. dashed) lines.
    }
    \label{oddcase}
\end{figure}

We will begin by showing that if a candidate line segment passes through the interior of $R_2$ then its interior contains no periodic points. Notice that such a candidate line segment is contained in the interior of a cylinder $C_1$ that makes an angle of $-\frac{\pi}{n}$ with the horizontal and in the union of two cylinders $C_2$ and $C_3$ that share a boundary saddle connection and make an angle of $-\frac{2\pi}{n}$ with the horizontal (see the left subfigure in Figure \ref{oddcase}). The claim that this candidate line segment contains no periodic points in its interior is now identical to the argument in the case of $n$ even.

It remains to consider the candidate line segment $\ell$ that is an edge of $R_2$ (see the right subfigure in Figure \ref{oddcase}). Let $Q'$ denote the midpoint of this candidate line segment, which is a Weierstrass point. 
We will show that $Q'$ is the only periodic point on the interior of $\ell$. As before, let $P$ denote $P_n$ and let $\overline{PQ'}$ denote the line contained in $\ell$ that begins at $P$, travels in the positive horizontal direction and ends at $Q'$. 

Notice that $\ell$ is entirely contained in a cylinder $C_1$ that makes an angle of $-\frac{\pi}{n}$ with the horizontal. Apply the element $r_n^{-1} s_n r_n$ of the Veech group, which shears the cylinders parallel to $C_1$. When $n$ is odd, any two parallel cylinders have equal moduli and so $r_n^{-1} s_n r_n$ acts on $C_1$ by performing a Dehn twist. In particular, the image $\overline{PQ'}$ under $r_n^{-1} s_n r_n$ remains in $C_1$ and becomes a line from $P$ to a new point $Q$ (see the right subfigure of Figure \ref{oddcase}). Let $\overline{PQ}$ denote this line.

% , letting $Q$ denote the image of $Q$ under $r_n^{-1} s_n r_n$, $\ell'$ is divided into one line segment $\overline{P_n Q'}$ contained in $R_1$ and another $\overline{Q' P_n}$ contained in $R_2$ (see the right subfigure of Figure \ref{oddcase}). 

It is easy to see that the line segment $\overline{P Q}$ is contained in two cylinders $C_2$ and $C_3$ that share a boundary saddle connection and make an angle of $-\frac{\pi}{n}$ with the vertical. Proceeding as in the case of $n$ even yields that $\overline{P Q}$ contains no periodic points in its interior and so the same must hold for $\overline{PQ'}$. Since $\ell$ is the union of $\overline{PQ'}$ and its image under the hyperelliptic involution, we have that $Q'$ is the only periodic point contained in the interior of $\ell$ as desired. \qedhere

% by the same reasoning $\overline{Q'P_n}$ has no periodic points on the interior.
% Thus, $Q'$ is the only periodic point in the interior of $\ell'$, so $Q$ is 
% the only periodic point on the candidate line segment $\ell$, as desired. 
% \ann{C: [Action] The notation here is inconsistent with the figure. Can we change the figure to have the same notation? \\ P: Update the notation to avoid constructions like $\overline{Q'P_n}$.}
\end{proof}

\begin{proof}[Proof of Corollary~\ref{main_corollary}]
By \cite[Theorem 2.6]{moeller_periodic_points} whenever a translation surface $(X, \omega)$ is not a translation cover of a torus there is a translation cover $\pi_{X_{min}}: (X, \omega) \rightarrow (X_{min}, \omega_{min})$ so that any translation cover with domain $(X, \omega)$ is a factor of $\pi_{X_{min}}$. Similarly, by \cite[Lemma 3.3]{marked_points}, there is a quadratic differential $(Q_{min}, q_{min})$ and a degree one or two (half)-translation cover $\pi: (X_{min}, \omega_{min}) \rightarrow (Q_{min}, q_{min})$ so that any half-translation cover is a factor of $\pi_{Q_{min}} := \pi \circ \pi_{X_{min}}$. By \cite[Theorem 3.6 and Lemma 3.8]{marked_points}, if $(p, q)$ are finitely blocked on $(X, \omega)$ then one of the following occurs:
\begin{enumerate}
    \item $p$ and $q$ are periodic points or zeros of $\omega$ and the blocking set may be taken to be the collection of all other distinct periodic points.
    \item Neither $p$ nor $q$ are periodic points or zeros of $\omega$, but $\pi_{Q_{min}}(p) = \pi_{Q_{min}}(q)$ and the blocking set may be taken to be the union of the periodic points with $\pi_{Q_{min}}^{-1}\left( \pi_{Q_{min}}(p) \right)$.
\end{enumerate}

Let $(X, \omega)$ now denote the regular $n$-gon or double $n$-gon surface. By Lemma \ref{min_cover}, $(X, \omega)$ is not a translation cover of a torus and $\pi_{X_{min}}$ is the identity. Since $(X, \omega) = (X_{min}, \omega_{min})$, the discussion above shows that $\pi_{Q_{min}}$ must be degree one or two. By uniqueness of $\pi_{Q_{min}}$, if $(X, \omega)$ admits any degree two map to a quadratic differential this map must be $\pi_{Q_{min}}$. Since $(X, \omega)$ is hyperelliptic, the quotient by the hyperelliptic involution is such a map and hence must be $\pi_{Q_{min}}$. %\ann{P: Added a shorter explanation. Chris' edit is commented out.}

%Therefore, $\pi_{Q_{min}}$ is the quotient by the hyperelliptic involution.\ann{C: [Action] Is it clear that pi is quotient by hyperelliptic? Citation? \\ P: Clarify the logic.} 
% We claim $\pi_{Q_{min}}$ is the quotient by the hyperelliptic involution $j$. Define $(Q,q)$ to be $(X,\omega)/j$. Then the projection $\pi_Q:X\to Q$ is a factor of $\pi_{Q_{min}}$. We have a commutative square
% $$
% \begin{tikzcd}
% X\arrow[r,"\pi_{X_{min}}"] \arrow[d,"\pi_Q"] &X_{min}\arrow[d,"\pi"]\\
% Q \arrow[r,"\pi'"] & Q_{min}
% \end{tikzcd}
% $$
% Since $\pi_{X_{min}}$ is the identity and $\pi,\pi_Q$ are degree $2$, thus $\pi' $ is also the identity, which proves the claim. \ann{C: added some of the justification we talked about.}

By \cite[Lemma 3.1]{marked_points}, all pairs $(p,q)$ where $p$ is not a zero of $\omega$ and $q$ is its image under the hyperelliptic involution are finitely blocked. (Note that the statement of \cite[Lemma 3.1]{marked_points} does not include the case when $p=q$ is a Weierstrass point that is not a zero of $\omega$. However, the proof is identical.) 

For the regular $n$-gon surface, the Weierstrass points are the center, midpoints of the sides, and when $4 \mid n$, the vertices of the regular $n$-gon that comprises the surface. For the double $n$-gon surface, the Weierstrass points are the midpoints of the sides and vertices of the two regular $n$-gons that comprise the surface. It is now clear that a zero of $\omega$ is never blocked from itself by the collection of Weierstrass points and hence is never finitely blocked from itself. 

It remains to show that if $p$ and $q$ are distinct points that are Weierstrass points or zeros of $\omega$, then they are not finitely blocked from each other. Since the blocking set would have to consist of the other Weierstrass points, convexity of the $n$-gons comprising $(X, \omega)$ and the explicit description of the Weierstrass points in the preceding paragraph shows that this is never possible.
Thus, the only finitely blocked points are the ones listed in the statement of the corollary.  

\end{proof}

\begin{proof}[Proof of Corollary~\ref{application}]
Let $T$ be a billiard table that unfolds to a translation surface $(X,\omega)$. Two points $p$ and $q$ on $T$ are finitely blocked if and only if every preimage of $p$ is finitely blocked from every preimage of $q$ on $(X,\omega)$. When $(X, \omega)$ is the regular $n$-gon or double $n$-gon surface, each point is finitely blocked from at most one other by Corollary \ref{main_corollary}. When $T$ is the $\left( \frac{\pi}{2}, \frac{\pi}{n}, \frac{(n-2)\pi}{2n} \right)$ triangle the only points on $T$ that have two or fewer preimages on $(X, \omega)$ are the vertices of angle $\frac{(n-2)\pi}{2n}$ and $\frac{\pi}{n}$, which, in the first case, unfolds to the zeros of $\omega$ and, in the second, to either the Weierstrass point $P_n$ when $n$ is even or to two points exchanged by the hyperelliptic involution when $n$ is odd. Since the zeros of $\omega$ are not finitely blocked from any other point, we see that the only possible pair of finitely blocked points on $T$ is the vertex of angle $\frac{\pi}{n}$ from itself. When $n$ is even, the preimage of this vertex on $(X, \omega)$ is $P_n$, which is finitely blocked from itself. When $n$ is odd, the preimage of this vertex consists of two points swapped by the hyperelliptic involution, and each point is not finitely blocked from itself.

%When $n$ is odd, the preimage of this vertex consists of two Weierstrass points, which are not finitely blocked from each other. 

% Notice that the preimages of the vertex of angle $\frac{\pi}{2}$ are the midpoints of all the edges of the regular $n$-gon(s) that comprise the surface, and these are not blocked from each other.  
% When $n$ is odd, the preimages of the vertex of angle $\frac{\pi}{n}$ are the two centers of the regular $n$-gons, and each is not finitely blocked from itself. 
% When $n$ is even, the preimage of the vertex of angle $\frac{\pi}{n}$ is the center of the regular $n$-gon, which is finitely blocked from itself. Thus, the only pair of finitely blocked points is the vertex of angle $\frac{\pi}{n}$ and itself.

%the collection of all Weierstrass points that are not zeros of $\omega$ with the Weierstrass point $P_n$ additionally exempted when $n$ is even. 
%By Corollary \ref{main_corollary} we see that, when $n$ is even, the vertex of angle $\frac{\pi}{n}$ is always blocked from itself by the other two vertices and that this is the only pair of finitely blocked points on $T$. \ann{P Sept 26: Does the last sentence have confusing phrasing?}

%\ann{P: A sentence should only begin with a word, not symbols. Also $p$ and $q$ are not defined. This should probably read Two points - $p$ and $q$ -  on $T$ are finitely blocked if and only if any preimage of $p$ is finitely blocked from any preimage of $q$ on $(X, \omega)$.}. The points listed are the only one with this property.
\end{proof}

%\ann{P: The citation for Apisa-Wright should include the arXiv number. There is a standard convention for citing arXiv preprints.}

% \begin{figure}
%     \centering
%     \includegraphics[width=100mm]{translation_cover.png}
%     \caption{Central cylinders which must be fixed. }
%     \label{cover}
% \end{figure}
% \ann{P: By the hyperelliptic involution}

\section{Proof of Lemma \ref{L:RatioSines}}\label{S:Ratio of Sines}

%\ann{R: Nov 4: I also changed the wording a little here. \\ P: I rephrased the sentences again. }

In this section we will show that for rational numbers $0 < \alpha < \beta \leq \frac{1}{2}$, $\frac{\sin(\pi \alpha)}{\sin(\pi \beta)}$ is rational if and only if $\alpha = \frac{1}{6}$ and $\beta = \frac{1}{2}$. McMullen stated this result in \cite[page 7]{ratio_sines} and indicated that its proof follows from an application of the bounds in the proof of \cite[Theorem 2.1]{ratio_sines}. Since we were unable to find an explicit proof in the literature, we provide it here. %\ann{P: Two sentences added}

%In this section we will show that for rational numbers $0 < \alpha < \beta \leq \frac{1}{2}$, $\frac{\sin(\pi \alpha)}{\sin(\pi \beta)}$ is rational if and only if $\alpha = \frac{1}{6}$ and $\beta = \frac{1}{2}$. McMullen stated this result in \cite[page 7]{ratio_sines} and indicated that its proof follows from \cite[Theorem 2.1]{ratio_sines}. Since the indicated proof does not appear more explicitly elsewhere, we provide it here. %\ann{P: Two sentences added}

For any positive integer $m$, let $\zeta_m := \mathrm{exp}\left( \frac{2\pi i}{m} \right)$ and \[ g(m) := 
  \begin{cases}
    m & \text{ $m \equiv 2 \text{ mod } 4$ odd} \\
    2m & \text{ $4 \mid m$} \\
   4m  & \text{ $m$ odd.} 
  \end{cases} \]
We begin with the following simple lemma. Throughout this section, if $p$ and $q$ are positive integers we will let $(p, q)$ denote their greatest common divisor. 

\begin{lemma}\label{L:MaximalReal}
For positive integers $k$ and $n$ with $(k, n) = 1$, $\Q\left( \sin\left( \frac{\pi k}{n} \right) \right)$ is the maximal real subfield of $\Q(\zeta_{g(n)})$.
\end{lemma}
\begin{proof}
It is well known that for positive integers $\ell$ and $m$ with $(\ell,m) = 1$, $\cos\left( \frac{2\pi \ell}{m} \right)$ generates the maximal real subfield of $\Q(\zeta_m)$. Notice that
\[ \sin\left( \frac{\pi k}{n} \right) = \cos\left( \frac{\pi}{2} - \frac{\pi k}{n} \right) = \cos \left( \frac{2\pi(n-2k)}{4n} \right). \]
Since $(k,n) = 1$, the only prime that divides both $n-2k$ and $4n$ is $2$. When $n$ is odd, we see that the numerator is odd, so $(n-2k, 4n) = 1$ and the claim holds. Similarly, when $4 \mid n$, $k$ is odd and so $\frac{n-2k}{4n} = \frac{\frac{n}{2} - k}{2n}$, where $\left(\frac{n}{2} - k, 2n\right) = 1$ since the numerator is odd. Finally, suppose that $n \equiv 2 \text{ mod } 4$, which in particular implies that $\Q(\zeta_n) = \Q(\zeta_{n/2})$. We see that $4 \mid n - 2k$ and that $8$ is the largest power of $2$ that divides $4n$. Therefore, when $\frac{n-2k}{4n}$ is put into lowest terms the denominator is either $n$ or $\frac{n}{2}$ and the result follows.
\end{proof}

Now let $\alpha = \frac{k_1}{n_1}$ and $\beta = \frac{k_2}{n_2}$ be rational numbers expressed in lowest terms where $0 < \alpha < \beta \leq \frac{1}{2}$  and so that $\frac{\sin(\pi \alpha)}{\sin(\pi \beta)}$ is rational. If $n_2 = 2$, then Niven's theorem implies that $\alpha = \frac{1}{6}$ and $\beta = \frac{1}{2}$. So suppose in order to deduce a contradiction that $n_2$ and hence also $n_1$ are greater than $2$. Let $N$ be the least common multiple of $n_1$ and $n_2$. 

\begin{lemma}\label{L:Equality}
$n_1 = n_2$
\end{lemma}
\begin{proof}
Suppose in order to deduce a contradiction that $n_1 \ne n_2$. Since $g$ is an injection onto the even integers, $\Q(\zeta_{g(n_1)}) \ne \Q(\zeta_{g(n_2)})$. The compositum of these two fields is $\Q(\zeta_{M})$ where $M$ is the least common multiple of $g(n_1)$ and $g(n_2)$. Since, by Lemma \ref{L:MaximalReal}, $\Q\left( \sin\left( \frac{\pi k_i}{n_i} \right) \right)$ is the maximal real subfield of $\Q(\zeta_{g(n_i)})$ for $i \in \{1, 2\}$ and since $\frac{\sin(\pi \alpha)}{\sin(\pi \beta)}$ is rational, we see that the maximal real subfields of $\Q(\zeta_{g(n_1)})$ and $\Q(\zeta_{g(n_2)})$ coincide. In general, if $E$ and $F$ are subfields of $\C$ so that, letting $K := E \cap F$, we have that $E/K$ and $F/K$ are Galois, then the compositum $EF/K$ is Galois and $\mathrm{Gal}(EF/K) \cong \mathrm{Gal}(E/K) \times \mathrm{Gal}(F/K)$. In our case, since we assumed that $n_i > 2$ for $i \in \{1,2\}$, $\Q(\zeta_{g(n_i)})$ is a degree two extension of its maximal real subfield $K$. This shows that $\Q(\zeta_M)$ is a degree four extension of $K$ and so $\phi(M) = 2\phi\left( g(n_1) \right) = 2\phi\left( g(n_2) \right)$ where $\phi$ denotes the Euler phi function.  Using that $g$ is injective and that $n_1 \ne n_2$, this implies that $M = 12k$ where $3 \nmid k$ for a positive integer $k$ and where, up to exchanging $n_1$ with $n_2$, $g(n_1) = 6k$ and $g(n_2) = 4k$. Since $\mathrm{Gal}\left( \Q(\zeta_M) / \Q \right)$ is isomorphic to $\left( \Z/M\Z \right)^\times$, we see that under the Galois correspondence $\Q(\zeta_{g(n_1)}) = \Q(\zeta_M^2)$ corresponds to the subgroup generated by $6k+1$ and the maximal real subfield of $\Q(\zeta_M)$ corresponds to the subgroup generated by $-1$. Again by the Galois correspondence, $\Q(\zeta_{g(n_2)}) = \Q(\zeta_M^3)$ must correspond to the subgroup generated by $-(6k+1)$. This implies that $-3(6k+1) \equiv 3 \text{ mod } 12k$, equivalently $6k \equiv 6 \text{ mod } 12k$, which implies that $k=1$. However, in this case $g(n_2) = 4$ which implies that $n_2 = 1$, a contradiction to the assumption that $n_2 > 2$.
\end{proof}

Following McMullen \cite[Proof of Theorem 2.1]{ratio_sines}, there is a constant $C_1$ such that $\frac{N}{n_1} \leq C_1$. By Lemma \ref{L:Equality} we may set $C_1 = 1$. By \cite[Proof of Theorem 2.1]{ratio_sines}, $\frac{1}{2} \leq \frac{5 \log(2N)}{N}$, which implies that $N < 45$. By Lemma \ref{L:Equality} we have the following
\[ \frac{\sin(\pi \alpha)}{\sin(\pi \beta)} = \frac{\zeta_{2N}^{k_2} - \zeta_{2N}^{-k_2}}{\zeta_{2N}^{k_1} - \zeta_{2N}^{-k_1}} = q \]
for some positive rational number $q \in \Q$. Since $k_2 > k_1$, this implies that $\zeta_{2N}$ is a root of the polynomial
\[ F(x) := x^{2k_2} - q x^{k_1+k_2} + q x^{k_2 - k_1} - 1.\] 
The minimal polynomial for $\zeta_{2N}$, which is the $(2N)$th cyclotomic polynomial, has degree $\phi(2N)$ and divides $F$.

\begin{lemma}
$F$ is the $(2N)$th cyclotomic polynomial. 
\end{lemma}
\begin{proof}
If not, then because $\phi(2N) \mid 2k_2$, it would follow that $2\phi(2N) \leq 2k_2$. Since $\frac{k_2}{N} = \beta < \frac{1}{2}$ we have, $\frac{\phi(2N)}{2N} < \frac{1}{4}$. Let $\Pi$ be the set of primes that divide $2N$, then we have
\[ \frac{\phi(2N)}{2N} = \prod_{p \in \Pi} \left( \frac{p-1}{p} \right) < \frac{1}{4} \]
For this inequality to hold, $N$ would need to have at least three prime factors aside from $2$. The smallest such number is $105$, which is greater than $45$.
\end{proof}

Since $k_2$ is coprime to $N$ and since $2k_2 = \phi(2N)$, we have that $(N, \phi(N)) = 1$. This implies that $N$ is squarefree and since $N > 2$, that $N$ is odd. Since $N  < 45$, $N$ is prime or $N \in \{15, 21, 33, 35, 39\}$. We can discard the cases of $N \in \{21, 39\}$ since $(N, \phi(N)) \ne 1$.  When $N$ is prime, the $(2N)$th cyclotomic polynomial is $\sum_{k=0}^{N-1} (-1)^{k} x^k$, which is never the same as $F$. The $(2N)$th cyclotomic polynomials for $N \in \{15, 33, 35\}$ all have more than four nonzero coefficients of monomial terms, so again they cannot be $F$ and we have a contradiction.

\bibliographystyle{halpha}
\bibliography{bibliography}

\begin{thebibliography}{LMW16}

\bibitem[Api]{Apisa:MarkedPointsGenus2}
Paul Apisa.
\newblock Periodic {P}oints in {G}enus {T}wo: {H}olomorphic {S}ections over
  {H}ilbert {M}odular {V}arieties, {T}eichmuller {D}ynamics, and {B}illiards.
\newblock preprint, arXiv:1710.05505v1 (2017).

\bibitem[Api20]{earle_kra}
Paul Apisa.
\newblock {${\rm GL}_2\mathbb{R}$}-invariant measures in marked strata: generic
  marked points, {E}arle-{K}ra for strata and illumination.
\newblock {\em Geom. Topol.}, 24(1):373--408, 2020.

\bibitem[AW17]{marked_points}
Paul Apisa and Alex Wright.
\newblock Marked points on translation surfaces, 2017, arXiv:1708.03411.

\bibitem[BM10]{bouw_moeller}
Irene~I. Bouw and Martin M\"{o}ller.
\newblock Teichm\"{u}ller curves, triangle groups, and {L}yapunov exponents.
\newblock {\em Ann. of Math. (2)}, 172(1):139--185, 2010.

\bibitem[Dan78]{Dani}
S.~G. Dani.
\newblock Invariant measures of horospherical flows on noncompact homogeneous
  spaces.
\newblock {\em Invent. Math.}, 47(2):101--138, 1978.

\bibitem[EM18]{EM}
Alex Eskin and Maryam Mirzakhani.
\newblock Invariant and stationary measures for the {${\rm SL}(2,\Bbb R)$}
  action on moduli space.
\newblock {\em Publ. Math. Inst. Hautes \'{E}tudes Sci.}, 127:95--324, 2018.

\bibitem[EMM15]{EMM}
Alex Eskin, Maryam Mirzakhani, and Amir Mohammadi.
\newblock Isolation, equidistribution, and orbit closures for the {${\rm
  SL}(2,\Bbb R)$} action on moduli space.
\newblock {\em Ann. of Math. (2)}, 182(2):673--721, 2015.

\bibitem[EMMW]{EMMW}
Alex Eskin, Curtis~T. McMullen, Ronen~E. Mukamel, and Alex Wright.
\newblock Billiards, quadrilaterals and moduli spaces.
\newblock preprint, to appear in J. Amer. Math. Soc.

\bibitem[GHS03]{affine_diffeomorphisms}
Eugene Gutkin, Pascal Hubert, and Thomas~A. Schmidt.
\newblock Affine diffeomorphisms of translation surfaces: periodic points,
  {F}uchsian groups, and arithmeticity.
\newblock {\em Ann. Sci. \'{E}cole Norm. Sup. (4)}, 36(6):847--866, 2003.

\bibitem[Hoo13]{grid_graphs}
W.~Patrick Hooper.
\newblock Grid graphs and lattice surfaces.
\newblock {\em Int. Math. Res. Not. IMRN}, (12):2657--2698, 2013.

\bibitem[HS06]{hubert_schmidt}
Pascal Hubert and Thomas~A. Schmidt.
\newblock An introduction to {V}eech surfaces.
\newblock In {\em Handbook of dynamical systems. {V}ol. 1{B}}, pages 501--526.
  Elsevier B. V., Amsterdam, 2006.

\bibitem[HST08]{HST}
P.~Hubert, M.~Schmoll, and S.~Troubetzkoy.
\newblock Modular fibers and illumination problems.
\newblock {\em Int. Math. Res. Not. IMRN}, (8):Art. ID rnn011, 42, 2008.

\bibitem[KM17]{KM}
Abhinav Kumar and Ronen~E. Mukamel.
\newblock Algebraic models and arithmetic geometry of {T}eichm\"{u}ller curves
  in genus two.
\newblock {\em Int. Math. Res. Not. IMRN}, (22):6894--6942, 2017.

\bibitem[LMW16]{LMW}
Samuel Leli\`evre, Thierry Monteil, and Barak Weiss.
\newblock Everything is illuminated.
\newblock {\em Geom. Topol.}, 20(3):1737--1762, 2016.

\bibitem[McM06]{ratio_sines}
Curtis~T. McMullen.
\newblock Teichm\"{u}ller curves in genus two: torsion divisors and ratios of
  sines.
\newblock {\em Invent. Math.}, 165(3):651--672, 2006.

\bibitem[McM07]{Mc5}
Curtis~T. McMullen.
\newblock Dynamics of {${\rm SL}_2(\Bbb R)$} over moduli space in genus two.
\newblock {\em Ann. of Math. (2)}, 165(2):397--456, 2007.

\bibitem[Mon05]{Mont1}
Thierry Monteil.
\newblock On the finite blocking property.
\newblock {\em Ann. Inst. Fourier (Grenoble)}, 55(4):1195--1217, 2005.

\bibitem[Mon09]{Mont2}
Thierry Monteil.
\newblock Finite blocking property versus pure periodicity.
\newblock {\em Ergodic Theory Dynam. Systems}, 29(3):983--996, 2009.

\bibitem[MT02]{mata}
Howard Masur and Serge Tabachnikov.
\newblock Rational billiards and flat structures.
\newblock In {\em Handbook of dynamical systems, {V}ol. 1{A}}, pages
  1015--1089. North-Holland, Amsterdam, 2002.

\bibitem[Mö06]{moeller_periodic_points}
Martin Möller.
\newblock Periodic points on {V}eech surfaces and the {M}ordell-{W}eil group
  over a {T}eichm\"{u}ller curve.
\newblock {\em Invent. Math.}, 165(3):633--649, 2006.

\bibitem[Vee89]{veech_eisenstein}
W.~A. Veech.
\newblock Teichm\"{u}ller curves in moduli space, {E}isenstein series and an
  application to triangular billiards.
\newblock {\em Invent. Math.}, 97(3):553--583, 1989.

\bibitem[Wol19]{Wolecki}
Amit Wolecki.
\newblock Illumination in rational billiards, 2019, arXiv:1905.09358.

\bibitem[Wri13]{wright_bouw_moller}
Alex Wright.
\newblock Schwarz triangle mappings and {T}eichm\"{u}ller curves: the
  {V}eech-{W}ard-{B}ouw-{M}\"{o}ller curves.
\newblock {\em Geom. Funct. Anal.}, 23(2):776--809, 2013.

\bibitem[ZK75]{unfolding}
A.~N. Zemljakov and A.~B. Katok.
\newblock Topological transitivity of billiards in polygons.
\newblock {\em Mat. Zametki}, 18(2):291--300, 1975.

\end{thebibliography}

\end{document}